\RequirePackage{fix-cm}
\documentclass[smallcondensed]{svjour3}     
\smartqed  
\usepackage{graphicx}
\usepackage{times,a4wide,mathrsfs}
\usepackage{amsmath}
\usepackage{amsfonts}

\newcommand{\C}{\mathbb{C}}
\newcommand{\ZZ}{\mathbb{Z}}

\newcommand{\QQ}{\mathbb{Q}}

\newcommand{\grif}{\hbox{Griff}}

\newcommand{\Hom}{\hbox{Hom}}

\newcommand{\gr}{\hbox{Gr}}

\newcommand{\wt}{\widetilde}
\newcommand{\ima}{\hbox{Im}}

\newcommand{\rom}{\romannumeral}

\newtheorem{convention}{Convention}

 \journalname{}
\begin{document}

\title{Surjectivity of cycle maps for singular varieties
}


\author{Robert Laterveer 
}


\institute{CNRS - IRMA, Universit\'e de Strasbourg \at
              7 rue Ren\'e Descartes \\
              67084 Strasbourg cedex\\
              France\\
              \email{laterv@math.unistra.fr}           
           }

\date{}

\maketitle

\begin{abstract}
A theorem of Jannsen asserts that if a smooth projective variety has injective cycle class maps, it has surjective cycle class maps. The object of this note is to present a version of Jannsen's theorem for singular quasi--projective varieties.

\end{abstract}

\section{Introduction}
\label{intro}

Let $X$ be a smooth complex projective variety. The cycle class maps
  \[ cl^i\colon A^iX_{\QQ}\to H^{2i}(X,\QQ)\]
  from Chow groups to singular cohomology have given rise to some of the most profound and fascinating conjectures in algebraic geometry: the Hodge conjecture (concerning the image of $cl^i$), and the Bloch--Beilinson conjectures (concerning the structure of the kernel of $cl^i$).
  
Since Mumford's work \cite{M}, it is well--known that if the Chow groups $A^iX_{\QQ}$ are ``small'' (in the sense of being supported on some subvariety), then also the singular cohomology groups are small (in the sense that they are supported on some subvariety).
The following result can be seen as an extreme instance of this general principle:

\begin{theorem}[Jannsen \cite{J}]\label{theo} Suppose $X$ is a smooth projective variety, such that
  \[    cl^i\colon A^iX_{\QQ}\to H^{2i}(X,\QQ)\]
   is injective for all $i$. Then there is an isomorphism
   \[ \bigoplus_{i\ge 0} cl^i\colon\ \ \bigoplus_{i\ge 0} A^iX_{\QQ}\ \stackrel{\cong}{\to}\ \bigoplus_{j\ge 0} H^j(X,\QQ)\ .\]
  (In particular, $H^{p,q}(X,\C)=0$ for all $p\not=q$.)
\end{theorem}

  This result can be proven using the Bloch--Srinivas method of decomposing the diagonal, and the formalism of correspondences \cite{BS}, \cite{Vo}.

 In this note, we look at the following question:
 
 \begin{question}\label{thequestion} If $X$ is only quasi--projective, and/or singular, in what sense is theorem \ref{theo} still true ?
 \end{question}
 
 This question has been treated by Lewis \cite[Corollary 0.3]{L2}, assuming a generalized version of the generalized Hodge conjecture holds. In this note, by contrast, we wanted to see how far we could get unconditionally. Our main result gives a version of theorem \ref{theo}, provided the singular locus of $X$ is not too large:
 
 \begin{theorem}
 Let $X$ be a quasi--projective variety of dimension $n$, and suppose there exists a compactification of $X$ with singular locus of dimension $\le {n+1\over 3}$. Suppose all cycle class maps are injective. Then the $cl_i$ induce an isomorphism
  \[\bigoplus_{i\ge 0}\colon\ \  \bigoplus_{i\ge 0} A_iX_{\QQ}\ \stackrel{\cong}{\to}\ \bigoplus_{\ell\ge 0} W_{-\ell} H_{\ell}(X,\QQ)\ .\]
  That is,
    \[ W_{-\ell} H_\ell X=\begin{cases} 0&\hbox{if\ $\ell$\ is\ odd;}\\
                                             \ima \, cl_i&\hbox{if\ $\ell=2i$.}\\
                                             \end{cases}\]  
   \end{theorem}

To prove this result, we adapt Jannsen's original method (i.e. the decomposition of the diagonal, plus the formalism of correspondences) to the singular and quasi--projective case. The decomposition of the diagonal goes through unchanged, except that in the quasi--projective case the boundary of a compactification appears in the decomposition (this is lemma \ref{diag}). As to correspondences: suppose $X$ is projective (not necessarily smooth) of dimension $n$. Then a correspondence, i.e. a cycle $C\in A_n(X\times X)_{\QQ}$, induces an action
  \[ C_\ast\colon H^i(X,\QQ)\ \to\ H_{2n-i}(X,\QQ)\ \]
  in a natural way (using the cap--product).
 Since $\Delta_\ast$ is just the canonical map (capping with the class of $X$), and since we can control the image of this canonical map in favourable cases (lemma \ref{durf}), we can conclude by looking at the action of the components occuring in the decomposition of $\Delta$.
 
We raise several questions in the course of this note. Indeed, in several respects even the smooth quasi--projective case is far from being as clear--cut and well--understood as the smooth projective case; a relation with ``Voisin's standard conjecture'' \cite{V0} naturally appears (remark \ref{voisin}).  

\begin{convention} In this note, the word {\sl variety\/} refers to a quasi--projective algebraic variety over $\C$.
\end{convention}

\section{The Bloch--Srinivas argument}
\label{sec:1}

\begin{definition} We will use $A_i$ to denote the Chow group of $i$--dimensional cycles, and $A^i$ to denote the Fulton--MacPherson operational Chow cohomology \cite{F}. 
By construction, $A^\ast$ acts on Chow groups $A_\ast$; in particular, for any projective $X$ there are natural maps
  \[  \begin{split} A^iX\ &\to\ \Hom\bigl( A_iX,\ZZ\bigr)\ ,\\
        b&\mapsto\ \hbox{deg}(b\cap -)\ .
        \end{split}\]

 We recall \cite{T} there are functorial cycle class maps
   \[  cl^i\colon A^iX\to \gr^W_{2i} H^{2i}(X,\QQ)\ ,\] 
   where $W$ is Deligne's weight filtration.                  
\end{definition}

\begin{definition} We will use the notation $A^i_c$ for ``compactly supported operational Chow cohomology''. This is defined as follows: for any quasi--projective $X$, let $X\subset\bar{X}$ be a compactification with boundary $D$. Define $A^i_cX$ by the exact sequence
  \[ 0\to A^i_cX\to A^i\bar{X}\to A^iD\ .\]
  This is independent of choice of $\bar{X}$ \cite{GS} (actually, $A^i_cX$ is what is denoted $R^0A^i X$ in \cite{GS}). Moreover $A^i_c$ is a contravariant functor for arbitrary morphisms \cite{GS}. There are natural maps
    \[ A^i_cX_{\QQ}\ \to\ \Hom\bigl( A_iX_{\QQ}, \QQ\bigr)\ ,\]
    defined by the above exact sequence. There are also functorial cycle class maps
      \[ cl^i\colon\ A^i_cX\ \to\ \gr^W_{2i} H^{2i}_c(X,\QQ)\ ,\]
      again defined by the above exact sequence.
\end{definition}  



\begin{definition} Let $X$ be a quasi--projective variety. Following Voisin \cite{V0}, \cite{Vo}, we say that $X$ has trivial Chow groups if the cycle class maps
  \[  cl_i\colon A_iX_{\QQ}\to H_{2i}(X,\QQ)\]
  are injective for all $i$.
  \end{definition}
  
  \begin{definition} Let $X$ be a quasi--projective variety. We say that 
    \[ \hbox{Niveau}\bigl( A_iX_{\QQ}\bigr)\le r\]
    if there exists a closed (i+r)--dimensional subvariety $Y\subset X$ such that $A_i(X\setminus Y)_{\QQ}=0$.
    \end{definition}

The key to the whole argument is the following decomposition lemma. This is the Bloch--Srinivas argument \cite{BS}; in his book, Bloch attributes this argument to Colliot--Th\'el\`ene \cite[appendix to lecture 1]{B}.

\begin{lemma}\label{diag} Let $\bar{X}$ be a projective variety of dimension $n$, and $X\subset\bar{X}$ the complement of a closed subvariety $D$. Suppose
  \[\hbox{Niveau}\bigl( A_iX_{\QQ}\bigr)\le r \ \ \ \ \hbox{for\ all\ $i$\ .}\]
  Then there is a decomposition of the diagonal
  \[ \Delta=\Delta_0+\Delta_1+\cdots+\Delta_n+\Gamma\ \ \in A_n(\bar{X}\times\bar{X})_{\QQ}\ ,\]
  where $\Delta_j$ is supported on $V_j\times W_j$, and $V_j\subset\bar{X}$ is of dimension $j+r$, $W_j\subset\bar{X}$ is of dimension $n-j$, and $\Gamma$ is supported on $\bar{X}\times D$.
\end{lemma}

\begin{proof} This is an application of the Bloch--Srinivas method \cite{BS}. 
We use the following two well--known lemmas:

\begin{lemma}\label{lim} Let $X$ and $Z$ be quasi--projective varieties, and suppose $Z$ is irreducible of dimension $n$. Then for any $i$
  \[  A_i(X_{k(Z)})\cong \varinjlim A_{i+n}(X\times U)\ ,\]
  where the limit is taken over opens $U\subset Z$.
  \end{lemma}
  
  \begin{proof} This is usually stated for smooth projective varieties \cite[appendix to Lecture 1]{B}. If one is brave, one goes checking in Quillen's work to see that the proof given in loc. cit. for the smooth case still goes on for singular varieties. Alternatively, take a resolution of singularities and reduce to the smooth case using the ``descent'' exact sequences, and the fact that $\varinjlim$ is an exact functor.
 \end{proof}
 
\begin{lemma}\label{inj} Let $X$ be a quasi--projective variety defined over a field $k$, and let $k\subset K$ be a field extension. Then
  \[A_i(X_k)_{\QQ}\to A_i(X_K)_{\QQ}\]
  is injective.
  \end{lemma}
  
  \begin{proof} This is usually stated for smooth varieties \cite[appendix to Lecture 1]{B}, but the same argument works in general: use lemma \ref{lim} to reduce to the case of a finite extension. For a finite extension, take a resolution of singularities; for smooth varieties, the existence of the norm implies the extension map is a split injection; by descent, the same is true for singular varieties.
 \end{proof}

Now we proceed with the proof of lemma \ref{diag}.
We can reduce to some subfield $k\subset\C$ which is finitely generated over its prime subfield. Consider the restriction
  \[ \Delta\in A_n(\bar{X}\times\bar{X})_{\QQ}\to A_n(X\times\bar{X})_{\QQ}\to A_0(X_{k(\bar{X})})_{\QQ}\ .\]
  The last group is supported in dimension $r$, so we get a rational equivalence
    \[ \Delta=\Delta_0+\Delta^1+\Gamma^1\ \ \in A_n(\bar{X}\times\bar{X})_{\QQ}\ ,\]
    where $\Delta_0$ is supported on $V_0\times\bar{X}$, where $V_0$ has dimension $r$, and $\Delta^1$ is supported on $\bar{X}\times W_1$ for some divisor $W_1$, and $\Gamma_1$ is supported on $D\times\bar{X}$. 
    
 Applying the same process to $\Delta^1$ and continuing inductively, we end up with a decomposition
    \[\Delta=\Delta_0+\Delta_1+\cdots+\Delta_n+\Gamma^\prime\ \ \in A_n(\bar{X}\times\bar{X})_{\QQ}\ ,\]
    where the $\Delta_j$ are as desired, but $\Gamma^\prime$ is supported on $D\times\bar{X}$. Taking the transpose and renumbering, we end up with a decomposition as desired.
\end{proof}

\begin{remark} In case $X$ is smooth projective, lemma \ref{diag} was proven in \cite{moi}, inspired by \cite{BS} and \cite{P}.
\end{remark}

\section{The smooth projective case}

The following is well--known.

\begin{proposition}[Jannsen \cite{J}, Kimura, Vial]\label{equiv} Let $X$ be a smooth projective variety of dimension $n$. The following are equivalent:

\item{(\rom1)} The groups $A_i^{alg}X_{\QQ}$ are $0$ for $i<{n\over 2}$;

\item{(\rom2)} $X$ has trivial Chow groups;

\item{(\rom3)} Niveau$(A_iX_{\QQ})\le 0$ for all $i$;
  
\item{(\rom4)} The cycle class maps induce a ring isomorphism  
  \[  A_\ast X_{\QQ}\ \stackrel{\cong}{\to}\ H_\ast(X,\QQ)\ ;\]
  
\item{(\rom5)} For any variety $Z$, and for any $i$, the product map induces an isomorphism
  \[  \bigoplus_{l+m=i} A_lX_{\QQ}\otimes A_mZ_{\QQ}\ \stackrel{\cong}{\to}\ A_i(X\times Z)_{\QQ}\ ;\]
  
\item{(\rom6)} For any variety $Z$, and for any $i,j$, the product map induces an isomorphism
  \[   \bigoplus_{l+m=i} A^lX_{\QQ}\otimes A^m(Z,j)_{\QQ}\ \stackrel{\cong}{\to}\ A^i(X\times Z,j)_{\QQ}\ \]  
  of higher Chow groups \cite{B2}.

  \end{proposition}
  
  \begin{proof} We will recall the proof, as a warm--up for what follows. To see that (\rom1)$\Rightarrow $(\rom2), we work inductively. First, 
  the hypothesis $A_0^{alg}X_{\QQ}=A_0^{hom}X_{\QQ}=0$ implies a decomposition of the diagonal
    \[ \Delta=X\times x+\Gamma^1\ \ \in A^n(X\times X)_{\QQ}\ ,\]
    with $\Gamma^1$ supported on $D\times X$, for some divisor $D$. Considering the action of $\Delta$ on $\grif^{n-1}X_{\QQ}$, we find
      \[\grif^{n-1}X_{\QQ}=0\ .\]
      (Indeed, the action factors over $\grif^{n-1}(\widetilde{D})_{\QQ}=0$, for some desingularisation $\widetilde{D}$.)
    Taken together with the hypothesis $A^{n-1}_{alg}X_{\QQ}=0$, we find that $A_1^{hom}X_{\QQ}=0$, and we continue likewise.
    
    After ${n\over 2}$ steps, we end up with a decomposition
    \[   \Delta=\Delta_0+\Delta_1+\cdots+\Delta_{ \lfloor {n\over 2}\rfloor}+\Gamma\ \ \in A^n(X\times X)_{\QQ}\ ,\]
    where $\Delta_j$ comes from $A^jX_{\QQ}\otimes A^{n-j}X_{\QQ}$, and $\Gamma$ is supported on $V\times X$ with $\dim V\le {n\over 2}$.
  We can apply this decomposition to $A_i^{hom}X_{\QQ}$ to check that (\rom2) holds. (Indeed, for $i\ge {n\over 2}$, the component $\Gamma$ does not act on $A_i^{hom}X_{\QQ}$ for dimension reasons; neither do the $\Delta_j$ act.)
  
  To see that (\rom2)$ \Rightarrow $(\rom3), remark that the cohomology groups $H^{2i}(X,\QQ)$ are finite--dimensional $\QQ$--vector spaces.
  
  To get the implication (\rom3)$\Rightarrow$(\rom4), let the decomposition of the diagonal act on the kernel and cokernel to see that both vanish.
  
  Now, let's prove the implication (\rom4)$\Rightarrow $(\rom5). Let $S\subset Z$ denote the singular locus, and let $\wt{Z}\to Z$ denote a resolution of singularities
  with exceptional divisor $E$.
 There is a commutative diagram with exact rows
   \[\begin{array}[c]{ccccccc}
       \to& A_i(X\times E)_{\QQ} &\to& A_i(X\times\wt{Z})_{\QQ}\oplus A_i(X\times S)_{\QQ}&\to& A_i(X\times Z)_{\QQ}&\to 0\\
             & \uparrow&&\uparrow&&\uparrow&\\
         \to&     \bigoplus_{l+m=i} A_lX_{\QQ}\otimes A_mE_{\QQ} &\to& \bigoplus_{l+m=i} A_lX_{\QQ}\otimes (A_m(\wt{Z})\oplus A_mS)_{\QQ}  &\to& \bigoplus_{l+m=i} A_lX_{\QQ}\otimes A_mZ_{\QQ}&\to 0\end{array}\]
    By noetherian induction, we are thus reduced to the case where $Z$ is smooth. Writing out a similar diagram for a compactification, we reduce to the case where $Z$ is smooth and projective.     
          
         Let's suppose now $Z$ is smooth projective, say of dimension $d$. We first prove surjectivity of the product map: 
   Take $c$ an element of $A_i(X\times Z)_{\QQ}$. We may suppose everything ($X$, $Z$, $c$ and the subvarieties supporting the $A_iX_{\QQ}$) is defined over a field $k\subset\C$ finitely generated over its prime subfield.  Consider what happens to $c$ under the restriction
   \[  c\in A_i(X\times Z)_{\QQ}\to A_{i-d}(X_{k(Z)})_{\QQ}=\varinjlim A_i(X\times U)_{\QQ}\ ,\]
   where the limit is taken over opens $U\subset Z$, and the equality is established in lemma \ref{lim}. 
   Since $k(Z)\subset\C$, lemma \ref{inj} implies that
     \[A_{i-d}(X_{k(Z)})_{\QQ}\to A_{i-d}(X_{\C})_{\QQ}\]
   is injective, so that 
     \[\hbox{Niveau}    \Bigl(  A_{i-d}(X_{k(Z)})_{\QQ}\Bigr)\le 0\ .\]
 It follows that the cycle $c$ can be written
   \[ c= b\times Z+c^\prime\ \ \in A_i(X\times Z)_{\QQ}\ ,\]
   where $b\in A_{i-n}X_{\QQ}$ and $c^\prime$ supported on $X\times Z^\prime$, for $Z^\prime\subset Z$ some divisor. By induction, the statement is true for $X\times Z^\prime$, and so we find that $c$ is a sum of product cycles as desired.
   
   Next, we prove the product map is injective. So let
   \[ p\colon\ \    \bigoplus_{l+m=i} A_lX_{\QQ}\otimes A_mZ_{\QQ}\ \stackrel{}{\to}\ A_i(X\times Z)_{\QQ}\ \]
   denote the product map, and let $a$ be an element in $\hbox{Ker}\, p$.
  We write 
    \[ a=\sum_{l+m=i} a_{l,m}\in \bigoplus_{l+m=i} A_lX_{\QQ}\otimes A_mZ_{\QQ}\ ,\]
   and we let $L$ be the maximum $l$ for which $a_{l,m}\not=0$. Hypothesis (\rom4) implies that $A_LX_{\QQ}$ is finite--dimensional, and that there is a perfect pairing
   \[  A_LX_{\QQ}\times A_{n-L}X_{\QQ}\ \to\ \QQ\ .\]
  Let $b_1,\ldots,b_r$ be a basis of $A_LX_{\QQ}$, and let $b_1^\vee,\ldots,b_r^\vee$ denote the dual basis of $A_{n-L}X_{\QQ}$.
  We write
    \[ a_{L,m}=\sum_j  c_j b_j\otimes d_j\ \ \in A_LX_{\QQ}\otimes A_mZ_{\QQ}\ .\]
  Since by hypothesis, $p(a)=0$, we have
    \[  p(a_{L,m})\cdot (b_j^\vee\times Z)=p(a)\cdot (b_j^\vee\times Z)=0\ \ \in A_m(X\times Z)_{\QQ}\ .\]  
    But $p(a_{L,m})\cdot (b_j^\vee\times Z)=c_j (\hbox{point})\times d_j$ projects to $c_j d_j\in A_m( Z)_{\QQ}$ under projection to the second factor, so we find
    \[ c_j d_j=0\ \ \in A_m(Z)_{\QQ}\ \ \forall j\ .\]
   But this means $a_{L,m}=0$; contradiction.
     
   To get that (\rom5) implies (\rom1): taking $Z=X$, one obtains a complete decomposition of the diagonal of $X$; having the diagonal act on $A_i^{alg}X_{\QQ}$, one obtains the required vanishing.
   
 It remains to establish an equivalence with (\rom6): using a commutative diagram extending the above diagram to the left (this exists, thanks to the notorious moving lemma for higher Chow groups \cite{B3}, \cite{Lev}), one is again reduced to the case $Z$ smooth projective. Now we use a result of Kimura \cite{Kim2} and Vial \cite[Theorem 5]{V2}, which states that (\rom2) is equivalent to the fact that the Chow motive of $X$ is a sum of twisted Lefschetz motives. Hence the motive of $X\times Z$ is a sum of twists of the motive of $Z$; as higher Chow groups only depend on the Chow motive, this implies (\rom6).
 
 Finally, (\rom6) $\Rightarrow$ (\rom5) is obvious, and we are done.
   
    \end{proof}

\begin{remark} Jannsen proved that properties (\rom2), (\rom3) and (\rom4) in proposition \ref{equiv} are equivalent \cite {J}.

The fact that it suffices to consider algebraically trivial cycles (i.e. point (\rom1) in proposition \ref{equiv}) is a particular instance of a more general phenomenon, discovered by Vial: if a morphism of Chow motives $f\colon N\to M$, with $N$ finite--dimensional, induces a surjection $A^{alg}_\ast(N)_{\QQ}\to A_\ast^{alg}(M)_{\QQ}$, then also $A_\ast^{hom}(N)_{\QQ}\to A_\ast^{hom}(M)_{\QQ}$ is surjective \cite[Theorem 7]{V3}. The above manifestation is just the case where $N$ is a Lefschetz motive.

Property (\rom5) is studied in depth in \cite{T2}, where it is called the ``Chow--K\"unneth property''. Notably, \cite[Theorem 4.1]{T2} generalizes the result of Kimura and Vial evoked in the above proof.

\end{remark}

\section{The smooth quasi--projective case}

In case $X$ is a smooth quasi--projective variety, the situation is not as well--understood as in the smooth projective case; the equivalences of proposition \ref{equiv} become difficult open problems. For instance, we raise the following questions:

\begin{question} Does the implication (\rom1)$\Rightarrow$(\rom2) of proposition \ref{equiv} still hold for $X$ smooth quasi--projective ?
\end{question}

\begin{question} Does the implication (\rom3)$\Rightarrow$(\rom2) of proposition \ref{equiv} still hold for $X$ smooth quasi--projective ?
\end{question}

\begin{remark}\label{voisin} For both questions, the answer is positive provided the ``Voisin standard conjecture'' \cite[Conjecture 0.6]{V0}, \cite{Vo} is true. Moreover, a positive answer to either question would imply the following result: if $X$ is smooth quasi--projective with trivial Chow groups, then any open $U\subset X$ has trivial Chow groups. As shown in \cite{V0}, this result (i.e. that ``having trivial Chow groups'' transfers from a variety to its open subsets)
would be a consequence of the truth of the Voisin standard conjecture.
\end{remark}

Here is what we can prove unconditionally:

\begin{proposition}\label{smoothtrivial} Let $X$ be a smooth quasi--projective variety with trivial Chow groups. Then the cycle class maps induce isomorphisms
  \[  \bigoplus_i A_i X_{\QQ}\ \stackrel{\cong}{\to}\ \bigoplus_\ell W_{-\ell} H_\ell (X,\QQ)\ .\]
  That is,
  \[ W_{-\ell} H_\ell (X,\QQ)=\begin{cases} 0&\hbox{if\ $\ell$\ is\ odd;}\\
                                             \ima\, cl_i&\hbox{if\ $\ell=2i$.}\\
                                             \end{cases}\]
            \end{proposition}
                                   
Since ``having trivial Chow groups'' obviously implies that the niveau of all Chow groups is $\le 0$, proposition \ref{smoothtrivial} follows from the following more general proposition:

\begin{proposition}\label{smoothcase} Let $X$ be a smooth quasi--projective variety of dimension $n$. Suppose
  \[\hbox{Niveau}(A_iX_{\QQ})\le 0\ \ \ \hbox{for\ all\ }i\ .\]
  Then 
    \[ W_{-\ell} H_\ell (X,\QQ)=\begin{cases} 0&\hbox{if\ $\ell$\ is\ odd;}\\
                                             \ima\, cl_i&\hbox{if\ $\ell=2i$.}\\
                                             \end{cases}\]
 Moreover,
     \[  cl^i\colon A^i_cX_{\QQ}\ \to\  \gr^W_{2i} H^{2i}_c(X,\QQ)\]
  is injective.
  \end{proposition}
  
 \begin{proof} Let $\tau\colon X\subset\bar{X}$ be a smooth compactification, with boundary $D$. From lemma \ref{diag}, we obtain a decomposition of the diagonal of $\bar{X}$
    \[ \Delta=\Delta_0+\Delta_1+\cdots+\Delta_n+\Gamma\ \ \in A_n(\bar{X}\times\bar{X})_{\QQ}\ ,\]
  where $\Delta_j$ is supported on $V_j\times W_j$, and $V_j$ (resp. $W_j$) is of dimension $j$ (resp.  $n-j$), and $\Gamma$ is supported on $\bar{X}\times D$. 
  Let $a\in \gr_F^{k} W_{-\ell} H_\ell(X,\C)$, with $\ell\not=-2k$. Let $\bar{a}\in \gr_F^k H_\ell(\bar{X},\C)$  be an element restricting to $a$.
  
  The action of $\Delta_j$ on $\bar{a}$ is $0$ for dimension reasons.
 (Indeed, let $\widetilde{V_j}$ and $\widetilde{W_j}$ denote resolutions of singularities.
 Then the action of $\Delta_j$ factors over 
   \[  H^{k+n,n-\ell-k}(\widetilde{V_j},\C)\]
   and
   \[ H^{k+n-j,n-\ell-k-j}(\widetilde{W_j},\C)\ ,\]
   and one of these groups is $0$ for dimension reasons.) 
   It follows that
     \[a=\tau^\ast\bar{a}=\tau^\ast\bigl( (\Delta_0)_\ast (\bar{a})+\cdots+(\Delta_n)_\ast(\bar{a})\bigr)=0\ \ \in \gr_F^k W_{-\ell} H_{\ell}(X,\C)\ ,\]
     so $\gr_F^k W_{-\ell} H_{\ell}(X,\C)$ for all $\ell\not=-2k$.
     In particular, for $\ell$ odd we find that
     \[ W_{-\ell} H_\ell(X,\C)=\bigoplus_j \gr_F^j W_{-\ell} H_\ell(X,\C)=0\ .\]

    Next, let $a\in  W_{-2i} H_{2i}(X,\QQ)\cap F^{-i}$. Using the polarisation on $H_{2i}(\bar{X},\QQ)$ (cf. \cite[]{Vo}), one finds there exists a Hodge class
 $ \bar{a}\in   H_{2i}(\bar{X},\QQ)$ which restricts to $a$ (i.e. $\tau^\ast\bar{a}=a$).   
   
 For $j\not=n-i$, the action of $\Delta_j$ on $\bar{a}$ is $0$ for dimension reasons. 
 (
 This is similar to the prior parenthesis: the action of $\Delta_j$ factors over 
   \[  H^{n-i,n-i}(\widetilde{V_j},\QQ)\]
   and
   \[ H^{n-i-j,n-i-j}(\widetilde{W_j},\QQ)\ ,\]
   and one of these groups is $0$ for dimension reasons.)
   For $j=n-i$, we have that
     \[  (\Delta_{n-i})_\ast \bar{a}\subset \ima\bigl( H_{2i}(W_{n-i})\to H_{2i}\bar{X}\bigr)\ \subset \ima cl_i\ .\]
    It follows that
      \[ a=\tau^\ast\bar{a}=\tau^\ast\bigl( (\Delta_{n-i})_\ast\bar{a}+ \Gamma_\ast\bar{a}\bigr) =\tau^\ast\bigl( (\Delta_{n-i})_\ast\bar{a}\bigr)\ \subset\ima cl_i\ .\]
      
   It remains to prove the statement for $cl^i$. Taking the transpose of all elements involved, we may suppose we have a decomposition
   \[ \Delta=\Delta_0+\Delta_1+\cdots+\Delta_n+\Gamma\ \ \in A_n(\bar{X}\times\bar{X})_{\QQ}\ ,\]
  where $\Delta_j$ are as before, but $\Gamma$ is now supported on $D\times\bar{X}$.       
      
      Let $a\in A^{i}_cX_{\QQ} $, and let $\bar{a}$ be the image of $a$ in $ A^{i}\bar{X}_{\QQ}$. The restriction of $\bar{a}$ to $D$ is $0$ (i.e., if $\psi\colon D\to\bar{X}$ denotes the inclusion, we have $\psi^\ast(\bar{a})=0\in A^{i}(D,\QQ)$), hence $\Gamma_\ast\bar{a}=0$.
      Just as above, the correspondence $\Delta_j$ does not act on $\bar{a}$ except for $j=i$, hence
        \[\bar{a}=(\Delta_i)_\ast\bar{a}\ .\]
   Suppose now $\bar{a}\in A^i_{hom}\bar{X}_{\QQ}$. The action of $\Delta_i$ on $A^i_{\hom}\bar{X}_{\QQ}$ factors over $A^0_{hom}(\widetilde{W_i})_{\QQ}$, which is $0$; it follows that $\bar{a}=0$, whence $a=0$.    
      \end{proof}

\section{The singular case}

In this section, we consider quasi--projective (possibly singular) varieties $X$. We prove our main result as promised in the introduction; this is
 a version of Jannsen's theorem for varieties whose singular locus is not too large:

\begin{theorem}\label {jannsen0} Let $X$ be a quasi--projective variety of dimension $n$, and suppose there is a compactification of $X$ with singular locus of dimension $\le {n+1\over 3}$. Suppose $X$ has trivial Chow groups. Then cycle class maps induce an isomorphism
  \[ \bigoplus_i A_iX_{\QQ}\ \stackrel{\cong}{\to}\ \bigoplus_\ell W_{-\ell} H_{\ell}(X,\QQ)\ .\]
  That is,
    \[ W_{-\ell} H_\ell (X,\QQ)=\begin{cases} 0&\hbox{if\ $\ell$\ is\ odd;}\\
                                             \ima\, cl_i&\hbox{if\ $\ell=2i$.}\\
                                             \end{cases}\]  
   \end{theorem}

This follows from the following more precise version:

\begin{theorem}\label{jannsen} Let $X$ be a quasi--projective variety of dimension $n$, and suppose a compactification of $X$ has singular locus of dimension $\le s$. Suppose
  \[\hbox{Niveau}(A_iX_{\QQ})\le 0\ \ \ \hbox{for\ all\ }i\ .\]
  Then 
    \[ W_{-\ell} H_\ell (X,\QQ)=\begin{cases} 0&\hbox{if\ $\ell$\ is\ odd;}\\
                                             \ima\, cl_i&\hbox{if\ $\ell=2i$,}\\
                                             \end{cases}\]  
    provided $\ell\in [0,n-s]\cup[2s,2n]$.  
      
Moreover,       
  \[ cl^i\colon A^i_cX_{\QQ}\to \gr^W_{2i} H^{2i}_c(X,\QQ)\]
 is injective in the range $i>s$.
  \end{theorem}

\begin{proof} Let $\tau\colon X\to \bar{X}$ denote the given compactification, with boundary
$D=\bar{X}\setminus X$. Applying lemma \ref{diag}, we find a decomposition of the diagonal of $\bar{X}$ of the form
    \[  \Delta=\Delta_0+\Delta_1+\cdots+\Delta_n\ \ \ \in A_n(\bar{X}\times \bar{X})_{\QQ}\ ,\]
   where $\Delta_j$ is supported on $V_j\times W_j$, and $V_j\subset \bar{X}$ has dimension $j$ and $W_j$ has dimension $n-j$. We can view $\Delta$ (and the $\Delta_j$) as a correspondence
   \[ \Delta_\ast\colon   H^i(\bar{X},\QQ)\to W_{i-2n} H_{2n-i}(\bar{X},\QQ)\ ,\]
   where for $\bar{a}\in H^i(\bar{X},\QQ)$, we define
   \[   \Delta_\ast (\bar{a}):=(\pi_2)_\ast \bigl((\pi_1^\ast \bar{a})\cap \Delta\bigr)\in W_{i-2n} H_{2n-i}(\bar{X},\QQ)\]
   (here $\pi_1$ resp. $\pi_2$ denotes projection on the first resp. second factor). It is easily checked that
   \[  \Delta_\ast(\bar{a})=\bar{a}\cap [\bar{X}]\in  H_{2n-i}(\bar{X},\QQ)\ .\]
   (Indeed, let $f\colon\wt{X}\to \bar{X}$ be a resolution of singularities, with projections $\wt{\pi}_1, \wt{\pi}_2$ from $\wt{X}\times\wt{X}$ to the two factors. Let $\wt{\Delta}$ denote the diagonal of $\wt{X}$, so that $\Delta=(f\times f)_\ast \wt{\Delta}$. Then
     \[   \begin{split} \Delta_\ast(\bar{a})&=(\pi_2)_\ast\bigl(\pi_1^\ast \bar{a}\cap\Delta\bigr)\\
                                        &=(\pi_2)_\ast (f\times f)_\ast \bigl( (f\times f)^\ast\pi_1^\ast \bar{a}\cap\wt{\Delta}\bigr)\\
                                        &=f_\ast (\wt{\pi}_2)_\ast \bigl( \wt{\pi}_1^\ast f^\ast \bar{a}\cap \wt{\Delta}\bigr)\\
                                        &=f_\ast \bigl(\wt{\Delta}_\ast(f^\ast \bar{a})\bigr)\\
                                            &=f_\ast(f^\ast \bar{a}\cap[\wt{X}])=\bar{a}\cap [\bar{X}].)\end{split}\]
  
 $\underline{\hbox{Case 1:}\ \ell\le n-s}$. Let $a\in \gr_F^p  W_{-\ell} H_\ell (X,\C)$, and let $\bar{a}\in \gr_F^p W_{-\ell} H_\ell (\bar{X},\C)$ be an element restricting to $a$. According to lemma \ref{durf} below, we can find $b\in \gr^W_{2n-\ell} H^{2n-\ell}(\bar{X},\C)$ such that 
  \[ \bar{a}=b\cap[\bar{X}]\ \ \in W_{-\ell} H_\ell(\bar{X},\C)\ .\]
  The ``Poincar\'e duality'' map from $H^{2i}$ to $H_{2n-i}$, being a map of Hodge structures, is strictly compatible with the Hodge filtration, so we may suppose $b\in \gr_F^{p+n}$.  
  Note that we have
    \[ a=\tau^\ast\bar{a}=\tau^\ast \Bigl( (\Delta_0)_\ast b+\cdots+ (\Delta_n)_\ast b\Bigr)\in \gr_F^p W_{-\ell}H_{\ell}(X,\C)\ .\]
 On the other hand, for dimension reasons we have
  \[  (\Delta_j)_\ast \gr_F^{p+n} \gr H^{2n-\ell}(\bar{X},\C)=0\ \ \ \hbox{unless\ \ }2n-\ell=2j=2(p+n)\ . \]
  It follows that $a=0$ if $\ell\not=-2p$; in particular $W_{-\ell} H_\ell (X,\QQ)=0$ for $\ell$ odd.
  
  Next, let $\ell=2i$ and consider $a\in W_{-2i} H_{2i}(X,\QQ)$. Using lemma \ref{durf}, we can find again $b\in  H^{2n-2i}(\bar{X},\QQ)$ such that
    \[ a=\tau^\ast\bigl( b\cap [\bar{X}]\bigr)=\tau^\ast\bigl( (\Delta_0)_\ast b+\cdots+ (\Delta_n)_\ast b\bigr)\in  H_{2i}(X,\QQ)  \ .\]
    But for reasons of dimension,
    \[  (\Delta_j)_\ast b=0\ \ \hbox{for\ }j\not=n-i\ ,\]
    and clearly
    \[  (\Delta_{n-i})_\ast b\in \ima\, cl_i\ .\]

   $\underline{\hbox{Case 2:} \ \ell\ge 2s}$. Let $S\subset X$ denote the singular locus, and $U=X\setminus S$ the non--singular locus. Then obviously
   \[ \hbox{Niveau}     \bigl( A_i(U)_{\QQ}\bigr)\le 0\ \ \ \ \hbox{for\ all\ $i$\  .}\]
 This implies (by proposition \ref{smoothcase} above) that 
   \[cl_i\colon A_iU_{\QQ}\to W_{-2i} H_{2i}(U,\QQ)\ \]
   is surjective for all $i$, and $W_{-\ell} H_\ell (U,\QQ)=0$ for $\ell$ odd. But the map
   \[  W_{-\ell} H_\ell (X,\QQ)\ \to\ W_{-\ell} H_\ell (U,\QQ)\]
   is an isomorphism for $\ell>2s$, so
   \[ W_{-\ell} H_\ell (X,\QQ)=0\ \ \ \hbox{if\ $\ell>2s$\ is\ odd\ .}\]
   
   Restriction induces a surjection
     \[  W_{-2i} H_{2i}(X,\QQ)\to W_{-2i} H_{2i}(U,\QQ)\]
  for reasons of weight; this fits into a commutative diagram with exact rows
    \[\begin{array}[c]{cccccc}
       A_iS_{\QQ}&\to& A_iX_{\QQ}&\to& A_iU_{\QQ}&\to 0\\
       \downarrow{cl_i}&&\downarrow{cl_i}&&\downarrow{cl_i}&\\
       W_{-2i} H_{2i}(S,\QQ)&\to&W_{-2i} H_{2i}(X,\QQ)&\to& W_{-2i} H_{2i}(U,\QQ)&\to 0\ .
       \end{array}\]
    The right vertical arrow is surjective, as we just noted, and the left vertical arrow is an isomorphism for $i\ge s$.   
    
    The "Moreover" part follows from the commutative diagram
    \[\begin{array}[c]{ccc}
        A^i_cU_{\QQ}&\to& A^i_cX\\
        \downarrow&&\downarrow\\
        \gr^W_{2i} H^{2i}_c(U,\QQ)&\to& \gr^W_{2i} H^{2i}_c(X,\QQ)\\
        \end{array}\]
     The horizontal maps are isomorphisms since $i>s$; the left vertical arrow is injective by proposition \ref{smoothcase}.

\begin{lemma}\label{durf} Let $X$ be a projective variety of dimension $n$, and with singular locus of dimension $\le s$. Then the natural map
  \[  \gr^W_{j} H^j(X,\QQ)\to W_{j-2n} H_{2n-j}(X,\QQ)\]
  is injective for $j\le n-s$, and surjective for $j\ge n+s$.
  \end{lemma}
  
 \begin{proof} Let $IH^jX$ denote middle--perversity intersection homology with rational coefficients. It follows from work of Durfee \cite{D} that
   \[  IH^jX=\begin{cases}   \gr^W_j H^j(X,\QQ), & j\ge n+s;\\
                                           W_{j-2n} H_{2n-j}(X,\QQ), & j\le n-s\ .
                                           \end{cases}\]
  It is well--known \cite{GM}, \cite{GM2} that the ``Poincar\'e duality'' map factors as
    \[  \gr^W_j H^j(X,\QQ) \to IH^jX \to W_{j-2n} H_{2n-j}(X,\QQ)\ .\]
 Moreover, it is known by work of Weber \cite{W} (cf. also \cite{HS}) that the first arrow is injective, and the second arrow surjective.
 \end{proof}   
 
 \end{proof}

 \begin{remark} It seems likely theorem \ref{jannsen} is true without any condition on the singular locus. This is proven by Lewis \cite[Corollary (0.2)]{L2}, under the assumption of (a generalized version of) the generalized Hodge conjecture.
\end{remark}

\begin{remark} Linear varieties (in the sense of \cite{T}) form a subclass of the class of varieties with trivial Chow groups. For a projective (possibly singular) linear 
variety $X$, Totaro has shown \cite{T} that
  \[\begin{split} cl_i\colon A_iX_{\QQ} &\to W_{-2i} H_{2i}(X,\QQ)\ ,\\
                        cl^i\colon A^iX_{\QQ}&\to \gr^W_{2i} H^{2i}(X,\QQ)   
                        \end{split}\]
            are isomorphisms for all $i$. (The first isomorphism is \cite[Theorem 3]{T}; the second isomorphism is obtained by combining the first isomorphism with \cite[Theorem 2]{T}.)
  \end{remark}

The argument in the proof of theorem \ref{jannsen} suggests the following question:

\begin{question} Let $X$ be any quasi--projective variety with 
  $\hbox{Niveau}(A_iX_{\QQ})\le 0$ for all $i$.
 Is it true that the natural map
 \[  \ima\Bigl( A^\ast X_{\QQ}\to A_{\ast} X_{\QQ}\Bigr)\ \stackrel{}{\to}\ \ima\Bigl(  H^\ast (X,\QQ)\to  H_\ast (X,\QQ)\Bigr) \]
 is an isomorphism ?
 \end{question}   

A partial answer is given by the following result: for $X$ projective, the right--hand side is generated by algebraic cycles.

\begin{proposition} Let $X$ be a projective variety of dimension $n$, and suppose 
    $\hbox{Niveau}(A_iX_{\QQ})\le 0\ \forall i$.
 Then 
 \[\ima\Bigl(  H^\ast (X,\QQ)\to H_\ast (X,\QQ)\Bigr)\]
 is generated by algebraic cycles. That is,
    \[  \ima\Bigl( H^\ell (X,\QQ)\to  H_{2n-\ell}(X,\QQ)\Bigr)=\begin{cases}  0&\hbox{if\ $\ell$\ is\ odd\ ;}\\
                                                                       \subset \ima\, cl_{n-i}&\hbox{if\ $\ell=2i$\ .}\\
                                                                       \end{cases}\]
                              \end{proposition}
                              
      \begin{proof} 
   First, let's suppose $\ell$ is odd. Then the vanishing of $\ima\bigl( H^\ell (X,\QQ)\to H_{2n-\ell}(X,\QQ)\bigr)$ follows from the following lemma:
   
   \begin{lemma} Set--up as in the proposition. Then
     \[    \gr_F^{p-n}\Bigl( \ima\bigl( H^\ell(X,\C)\to H_{2n-\ell}(X,\C)\bigr)\Bigr)=0\ \ \hbox{for\ }\ell\not=2p\ .\]
     \end{lemma}
     
     \begin{proof}{(of the lemma)} By strict compatibility of the Hodge filtration, the group in the statement of the lemma is the same as
     \[ \ima\bigl( \gr_F^p H^\ell(X,\C)\to \gr_F^{p-n}H_{2n-\ell}(X,\C)\bigr)\ .\]
     This is the same as
       \[ \Delta_\ast \gr_F^p H^\ell(X,\C)\ ,\]
      where correspondences act as defined in the proof of proposition \ref{jannsen}. Now we apply the decomposition of $\Delta$ given by lemma \ref{diag}.
      The action of the component $\Delta_j$ factors as follows:
     \[ \begin{array}[c]{ccc}
          \cdots&&\\
          \uparrow&&\downarrow\\
          \gr_F^p H^\ell(\widetilde{V_j},\C)&& \gr_F^{p-n} H_{2n-\ell}(\widetilde{W_j},\C)\\
          \uparrow&&\downarrow\\
          \gr_F^p H^\ell (X,\C)&\stackrel{(\Delta_j)_\ast}{\to}& \gr_F^{p-n} H_{2n-\ell}(X,\C)\\
          \end{array}\]
          
     Since $\dim \widetilde{V_j}=j$, the upper left group vanishes for $p>j$; likewise, since $\dim \widetilde{W_j}=n-j$, the upper right group vanishes for $p<j$. It follows that the only non--trivial action is for $p=j$. But the group $\gr^j_F H^\ell(\widetilde{V_j},\C)$ vanishes unless $\ell=2j$.     
  \end{proof}     
  
  It remains to treat the case $\ell=2i$. Let $a\in H^{2i}(X,\QQ)$.
  From the proof of the above lemma, we find that
    \[a\cap [X]=\Delta_\ast a=(\Delta_i)_\ast a\ \ \in H_{2n-2i}(X,\QQ)\]
   (indeed, for $j\not=i$, the action
     \[  (\Delta_j)_\ast H^{2i}(X,\C)=0\ \ \in H_{2n-2i}(X,\C)\ ,\]
     since it is $0$ on each $\gr_F^p$. But
       \[   \begin{array}[c]{ccc}
                            H^{2i}(X,\QQ)&\stackrel{(\Delta_j)_\ast}{\to}  & H_{2n-2i}(X,\QQ)\\
                            \downarrow&&\downarrow\\
                            H^{2i}(X,\C)&\stackrel{(\Delta_j)_\ast}{\to}  & H_{2n-2i}(X,\C)\\    
                            \end{array}\]
                            commutes, so that also $(\Delta_j)_\ast H^{2i}(X,\QQ)=0$ for $j\not=i$.)

                        But
       \[  (\Delta_i)_\ast H^{2i}X\subset \ima\bigl( H_{2n-2i}(W_i,\QQ))\to H_{2n-2i}(X,\QQ)\bigr)\subset \ima\, cl_{n-i}\ .\]
       \end{proof}

 \begin{remark} A result analogous to Jannsen's theorem is proven by Esnault--Levine \cite{EL}. They prove that if $X$ is smooth projective such that
 all cycle class maps into Deligne cohomology are injective, these cycle class maps are surjective (this is reproven, and rendered more precise, in \cite[Theorem 4]{V2}.) 
 Lewis extends this to singular and quasi--projective varieties, again assuming (a generalized version of) the generalized Hodge conjecture \cite[Corollary (0.3)]{L2}.
 It would be interesting to try whether the approach of the present note can be applied to this problem; I haven't looked into this.
 \end{remark} 
 
 \begin{remark} The argument of the proof of theorem \ref{jannsen} can also be used to obtain a new version of Mumford's theorem for singular varieties, plus a verification of the Hodge conjecture for certain singular varieties \cite{moi2}.
 \end{remark}

     
      


\begin{acknowledgements}
This note was stimulated by the Strasbourg ``groupe de travail'' based on the monograph \cite{Vo}. It is a pleasure to thank the participants of this groupe de travail for a very pleasant atmosphere and stimulating interactions.
\end{acknowledgements}


\begin{thebibliography}{}




\bibitem{B} S. Bloch, Lectures on algebraic cycles, Duke Univ. Press 1980,

\bibitem{B2} S. Bloch, Algebraic cycles and higher K--theory, Advances in Math. vol. 61 (1986), 267---304,

\bibitem{B3} S. Bloch, The moving lemma for higher Chow groups, J. Alg. Geom. 3 (1994), 537---568,

\bibitem{BS} S. Bloch and V. Srinivas, Remarks on correspondences and algebraic cycles, American Journal of Mathematics Vol. 105, No 5 (1983), 1235---1253,

\bibitem{D} A. Durfee, Intersection homology Betti numbers, Proc. Amer. Math. Soc. 123 (1995), 989---993,

\bibitem{EL} H. Esnault and M. Levine, Surjectivity of cycle maps, in: Journ\'ees de G\'eom\'etrie alg\'ebrique d'Orsay, Ast\'erisque 218 (1993),

\bibitem{F} W. Fulton, Intersection theory, Springer--Verlag Berlin Heidelberg New York 1984,



\bibitem{GS} H. Gillet and C. Soul\'e, Descent, motives and K--theory, J. Reine Angew. Math. 478 (1996), 127---176,

\bibitem{GM} M. Goresky and R. MacPherson, Intersection homology theory, Topology 19 (1980), 135---162,

\bibitem{GM2} M. Goresky and R. MacPherson, Intersection homology II, Inv. Math. 71 (1983), 77---129,

\bibitem{HS} M. Hanamura and M. Saito, Weight filtrations on the cohomology of algebraic varieties, arxiv 0605603v2,

\bibitem{J} U. Jannsen, Mixed motives and algebraic K--theory, Springer Lecture Notes in Mathematics 1400 (1990),

\bibitem{J2} U. Jannsen, Motivic sheaves and filtrations on Chow groups, in: Motives (Jannsen et alii, eds.), Proceedings of Symposia in Pure Mathematics Vol. 55 (1994), Part 1,  




\bibitem{Kim2} S. Kimura, Surjectivity of the cycle map for Chow motives, in: Motives and algebraic cycles, Fields Institute Comm. vol. 56, Amer. Math. Soc., Providence 2009,



\bibitem{moi} R. Laterveer, Algebraic varieties with small Chow groups, J. Math. Kyoto Univ. Vol. 38 No 4 (1998), 673---694,

\bibitem{moi2} R. Laterveer, Correspondences and singular varieties, to appear in Monatshefte f\"ur Mathematik,

\bibitem{Lev} M. Levine, Techniques of localization in the theory of algebraic cycles, J. Alg. Geom. 10
(2001), 299---363,


\bibitem{L2} J. Lewis, A generalization of Mumford's theorem, II, Illinois Journal of Mathematics, Vol. 39 No 2 (1995), 288---304,

\bibitem{M} D. Mumford, Rational equivalence of $0$--cycles on surfaces, J. Math. Kyoto Univ. Vol. 9 No 2 (1969), 195---204,

\bibitem{P} K. Paranjape, Cohomological and cycle--theoretic connectivity, Annals of Mathematics 140 (1994), 641---660,





\bibitem{T} B. Totaro, Chow groups, Chow cohomology, and linear varieties, Forum of Mathematics, Sigma (2014), vol. 1, e1,

\bibitem{T2} B. Totaro, The motive of a classifying space, preprint,


\bibitem{V2} C. Vial, Projectors on the intermediate algebraic Jacobians, New York J. Math. 19 (2013), 793---822,

\bibitem{V3} C. Vial, Remarks on motives of abelian type, arxiv:1112.1080v2,



\bibitem{V0} C. Voisin, The generalized Hodge and Bloch conjectures are equivalent for general complete intersections, Annales scientifiques de l'ENS 46, fascicule 3 (2013), 449---475,


\bibitem{Vo} C. Voisin, Chow Rings, Decomposition of the Diagonal, and the Topology of Families, Princeton University Press, Princeton and Oxford, 2014,

\bibitem{W} A. Weber, Pure homology of algebraic varieties, Topology 43 (2004), 635---644.





\end{thebibliography}


\end{document}